\documentclass[11pt]{amsart}

\usepackage{hyperref}
\usepackage{amsfonts,mathrsfs,bbm,latexsym,rawfonts,amsmath,amssymb,amsthm,epsfig}
\usepackage{fullpage, setspace, color}
\newtheorem{theorem}{Theorem}[section]

\newtheorem{remark}[theorem]{Remark}
\newtheorem{lemma}[theorem]{Lemma}

\numberwithin{equation}{section}

 \newcommand{\al}{\alpha}
 \newcommand{\be}{\beta}
 
 \newcommand{\Ld}{\Lambda}

 \newcommand{\ep}{\varepsilon}
 \newcommand{\Si}{\Sigma}
 \newcommand{\si}{\sigma}
 \newcommand{\om}{\omega}
 \newcommand{\Om}{\Omega}
 \newcommand{\ga}{\gamma}
 \newcommand{\Ga}{\Gamma}

 \newcommand{\sch}{Schr\"odinger }

 \newcommand{\A}{\textbf{A}}

 \newcommand{\Real}{\mathbb{R}}

 \newcommand{\abs}[1]{\vert#1\vert}
 \newcommand{\norm}[1]{\Vert#1\Vert}
 \newcommand{\Norm}[1]{\left\Vert#1\right\Vert}
 \def\<{\left\langle} \def\>{\right\rangle}
 \def\({\left(} \def\){\right)}
 \newcommand{\n}{\nabla}
 \newcommand{\p}{\partial}

 \renewcommand{\H}{\textbf{H}}

\allowdisplaybreaks

\begin{document}

\title{Skew Mean Curvature Flow }

\author{Chong Song, Jun Sun\\
\\
\\
\emph{In memory of Professor Weiyue Ding}
}

\address{School of Mathematical Sciences, Xiamen University, Xiamen, 361005, P.R.China.}
\email{songchong@xmu.edu.cn}

\address{School of Mathematics and Statistics, Wuhan University, Wuhan, 430072, P.R.China.}
\email{sunjun@whu.edu.cn}

\thanks {The first author is supported by Fundamental Research Funds for the Central University (No. 20720170009); the second author is supported by NSFC (No. 11401440).}

\date{\today}

\begin{abstract}
The skew mean curvature flow(SMCF), which origins from the study of fluid dynamics, describes the evolution of a codimension two submanifold along its binormal direction. We study the basic properties of the SMCF and prove the existence of a short-time solution to the initial value problem of the SMCF of compact surfaces in Euclidean space $\Real^4$. A Sobolev-type embedding theorem for the second fundamental forms of two dimensional surfaces is also proved, which might be of independent interest.
\end{abstract}

\maketitle

\section{Introduction}

The \emph{skew mean curvature flow}(SMCF) is a geometric flow which evolves a codimension two submanifold along its binormal direction with a speed given by its mean curvature. Specifically, suppose $\Si$ is an $n$ dimensional oriented manifold and $(\overline{M},\bar{g})$ is an $(n+2)$ dimensional oriented Riemannian manifold. Let $I=[0,T)$ be an interval and $F:I\times \Si\to \overline{M}$ be a family of immersions. For each $t\in I$, denote the submanifold by $\Si_t=F(t, \Si)$ and its mean curvature by $\textbf{H}(F)$. The normal bundle $N\Si_t$ of the submanifold is a rank two vector bundle with a naturally induced complex structure $J(F)$ which simply rotates a vector in the normal space by $\pi/2$ positively. More precisely, for any point $y=F(t,x)\in \Si_t$ and normal vector $\nu\in N_y\Si_t$, we require $J(F)\nu\perp\nu$ and $\bar{\om}(F_*(e_1), \cdots, F_*(e_n), \nu, J(F)\nu)>0$, where $\bar{\om}$ is the volume form of $\overline{M}$ and $e_1, \cdots, e_n$ is an oriented basis of $\Si$. We shall call the binormal vector $J(F)\textbf{H}(F)$ the \emph{skew mean curvature vector} and the SMCF is defined by
\begin{equation}\label{e:SMCF}
         \frac{\partial F}{\partial t}=J(F)\textbf{H}(F).
\end{equation}

In particular, the one dimensional SMCF in the Euclidean space $\Real^3$ is just the well-known \emph{vortex filament equation}(VFE)
\begin{equation}\label{e:darios}
  \p_t \gamma = \p_s\gamma\times \p_s^2\gamma,
\end{equation}
where $\ga$ is a time dependent space curve, $s$ is its arc-length parameter and $\times$ denotes the cross product in $\Real^3$. To see this, let $\{\mathbf{t}, \mathbf{n}, \mathbf{b}\}$ be the Frenet frame and $k$ be the curvature of the curve $\gamma(t, \cdot)$. Then the mean curvature vector is $\textbf{H}(\ga)=\p_s^2\gamma=k\mathbf{n}$ and the complex structure $J(\ga)=\p_s\gamma\times$ rotates $\mathbf{n}$ to the binormal vector $\mathbf{b}$. Thus equation (\ref{e:darios}) is equivalent to
\begin{equation*}
  \p_t \gamma = J(\ga)\textbf{H}(\ga)=k\mathbf{b}.
\end{equation*}
The VFE was first discovered by Da Rios~\cite{Da} in 1906 in the study of the free motion of a vortex filament. A key feature of the VFE is that, by a so-called Hasimoto transformation~\cite{Hasimoto}, it is equivalent to a complex-valued cubic \sch equation which is completely integrable. Thus it admits soliton solutions and has very rich structure.

General SMCF naturally arises in higher dimensional hydrodynamics. A singular vortex in a fluid is called a \emph{vortex membrane} if it is supported on a codimension two subset. The law of locally induced motion of a vortex membrane can by deduced from the Euler equation by applying the Biot-Savart formula. In 2012, Shashikanth~\cite{Sh} first investigated the motion of a vortex membrane in $\Real^4$ and showed that it is governed by the two dimensional SMCF. Khesin~\cite{Kh} generalized this conclusion to any dimensional vortex membranes in Euclidean spaces and gave the formal definition of the SMCF which we apply here.

The SMCF also emerges in the study of asymptotic dynamics of vortices in the context of superfluidity and superconductivity. Here the model is usually given by a PDE of a complex field and the vortices are just zero sets of the solutions. For example, for the Ginzburg-Landau heat flow, it was shown that asymptotically the energy concentrates on the codimension two vortices which moves along the mean curvature flow~\cite{Lin98,JS99,BOS,JL}. Similar phenomena are observed for wave and \sch type PDEs. In particular, for the Gross-Pitaevskii equation which models the wave function associated with a Bose-Einstein condensate, physics evidences indicate that the vortices would evolve along the SMCF. This was first verified by Lin~\cite{LinT} for the vortex filaments in three space dimensions. For higher dimensions, Jerrard~\cite{Jerrard} proved this conjecture when the initial singular set is a codimension two sphere with multiplicity one in 2002. It is worth mentioning that he also proposed a notion of weak solution to the SMCF.

Besides its physical significance, the SMCF is a rather canonical geometric flow for codimension two submanifolds which can be viewed as the Schr\"odinger-type counterpart of the well-known mean curvature flow(MCF). In fact, the SMCF has a notable Hamiltonian structure and is volume preserving. The infinite dimensional space of codimension two immersions of a Riemannian manifold admits a generalized Marsden-Weinstein symplectic structure~\cite{MW}, and the SMCF turns out to be the Hamiltonian flow of the volume functional on this space. This fact was first noted by Haller and Vizman~\cite{HV} where they studied the non-linear Grassmannians. For completeness, we include a detailed explanation in Section~\ref{s:sym} below.

The SMCF is also related to another important Hamiltonian flow, namely, the \emph{\sch flow}~\cite{DW1,TU,Ding}. The \sch flow stems from the study of ferromagnetism and is the Hamiltonian flow of the energy functional defined on the space of maps from a Riemannian manifold to a symplectic manifold. It is well-known that, if a curve satisfies the VFE~(\ref{e:darios}), then its Gauss map satisfies the one dimensional \sch flow on the standard sphere. In fact, this relation also holds true for higher dimensions. In~\cite{S}, it is shown that the Gauss map of an $n$ dimensional SMCF in $\Real^{n+2}$, which maps from the submanifold to the Grassmannian $G(n, n+2)$ (which is K\"ahler), satisfies a \sch flow equation. However, different from the one dimensional case, the metric of the underlying manifold of the \sch flow is evolving along the SMCF for $n\ge 2$.

Despite of these various sources of interest, little is known about the SMCF except the one dimensional case, i.e. classical VFE. There is a vast body of literatures on the VFE and its dynamics are well-understood. For example, the global well-posedness of the VFE is obtained in~\cite{NT} by method of regularization and recently in \cite{JS15} in a different framework. As a natural generalization of the VFE in both higher dimensions and Riemannian geometry, the SMCF has drawn more and more attention in recent years. Gomez proved the global existence of one dimensional SMCF in a general three dimensional Riemannian manifold in his thesis~\cite{Go}, which also contains some partial results on the Hasimoto transformation of two dimensional SMCF. As far as we know, Lin and his collaborators has an undergoing project on the energy conserved motion~\cite{Lin09,LW10}, which includes the SMCF of surfaces in $\Real^4$. The SMCF is also independently proposed under the name of star mean curvature flow by Terng~\cite{Terng}. However, basic issues like local well-posedness of the general SMCF is still open.

The current paper takes a first step towards the research of higher dimensional SMCF. We explore the basic properties of general SMCF and show the local existence of two dimensional SMCF of compact surfaces in the Euclidean space $\Real^4$. The uniqueness of SMCF will be addressed in another sequel.

Now let's state our main results and explain the difficulties in the proof. Suppose $\Si$ is a two dimensional oriented compact surface and $F_0$ is a smooth immersion from $\Si$ to $\Real^4$. We consider the initial value problem
\begin{equation}\label{e:SMCF1}
\left\{\begin{aligned}
         &\frac{\partial F}{\partial t}=J(F)\textbf{H}(F), \\
         &F(0,\cdot)=F_0.
\end{aligned}\right.
\end{equation}
Let $\textbf{A}_0$ denote the second fundamental form of the immersed surface $F_0(\Si)$, we can define a Sobolev-type norm $\norm{\textbf{A}_0}_{H^{2,2}}$ by the induced metric and normal connection(see~(\ref{e:sobolev-norm}) for details). Our main result is

\begin{theorem}\label{t:main}
Suppose $\Si$ is a two dimensional oriented compact surface. For any smooth immersion $F_0:\Sigma\to\mathbb{R}^4$, the SMCF (\ref{e:SMCF1}) admits a smooth local solution $F\in
C^\infty([0, T)\times\Si)$, where the time $T$ only depends on $\norm{\textbf{A}_0}_{H^{2,2}}$ and the volume of $F_0(\Si)$.
\end{theorem}

\begin{remark}
The existence result actually holds for less smooth initial data, e.g. for $W_{loc}^{4,2}$-immersions with $H^{2,2}$-bounded second fundamental forms. Recall that the optimal well-posedness result of the \sch flow from 2 dimensional manifolds is established in the $W^{3,2}$ Sobolev space~\cite{DW2,Mc}. Since the Gauss map of a solution to the SMCF satisfies a \sch flow with varying metric~\cite{S}, the requirement of $W^{4,2}$-regularity of the initial data for the SMCF is consistent.
\end{remark}

\begin{remark}
With our method, it is not hard to show that same results hold for a general 4 dimensional ambient Riemannian manifold with bounded geometry. However, Theorem~\ref{t:main1} below and hence same strategy fails for SMCF of submanifolds of dimension larger than two.
\end{remark}

The above result might seem standard when compared to the well-posedness of MCF. However, due to the skew-symmetric operator $J$, the SMCF is a (degenerate) \sch type system which has a totally different character from parabolic equations from the perspective of PDEs. Generally, there is no standard theory of existence of solutions and the DeTurck trick does not apply. One approach to obtain existence results of \sch type systems is to use a parabolic approximation, which proved to be successful in the study of the \sch flow~\cite{DW2} and other kind of Hamiltonian flows~\cite{SY}. Here we adopt the same strategy.

More precisely, we consider the perturbed system for a small real number $\varepsilon>0$
 \begin{equation}\label{e:pSMCF}
\left\{\begin{aligned}
         &\frac{\partial F}{\partial t}=J\textbf{H}+\varepsilon\textbf{H}, \\
         &F(0, \cdot)=F_0.
\end{aligned}\right.
\end{equation}
The system~(\ref{e:pSMCF}) is weakly parabolic and behaves similar as the MCF. By applying the DeTurck trick and standard parabolic theories, it is easy to find that (\ref{e:pSMCF}) admits a local solution $F_\ep$ on some time interval $[0,T_\ep)$ for every $\ep>0$. Next we need to show that $F_\ep$ converges to a solution of the original SMCF (\ref{e:SMCF1}) as $\ep\to 0$. Thus the problem is reduced to deriving uniform estimates of $F_\ep$ as well as a lower bound for the lifespan $T_\ep$.

To obtain uniform estimates for the perturbed flow, a key ingredient is a Gagliardo-Nirenberg interpolation inequalities on vector bundles(cf. \cite{DW2}). However, along the SMCF, since the induced metric of the underlying manifold is varying , the Sobolev constants can not by chosen uniformly in general. Our solution is to adopt the interpolation inequality of tensors which is independent of the metric and proven by Hamilton~\cite{Ha} in the study of Ricci flow. It will be used to derive a Gronwall type inequality for the Sobolev norms of the second fundamental forms. But this method relies on a uniform a priori $C_0$-estimate of the second fundamental form. The following uniform $C_0$-estimate of the second fundamental forms of surfaces plays a crucial role in our proof, which might be of independent interest.

\begin{theorem}\label{t:main1}
Given positive numbers $B$ and $m$, there exists a constant $C(B,m)$, depending only on $B$ and $m$, such that for any immersed compact surface $\Si^2\subset\Real^4$ satisfying
\begin{equation*}
\norm{\textbf{A}}_{H^{2,2}}\le B \text{~~and~~} |\Si|\ge m,
\end{equation*}
there holds
\begin{equation*}
\norm{\textbf{A}}_{C^0}\le C(B,m).
\end{equation*}
\end{theorem}

\begin{remark}
The above theorem actually holds for two dimensional immersed surfaces in any higher dimensional Euclidean spaces. In order to prove this theorem, we establish a compactness theorem of surfaces with bounded Sobolev norms of the second fundamental forms, which is a generalized version of the compactness theorem due to Langer~\cite{Langer}. See Theorem~\ref{t:compactness1}, Theorem~\ref{t:compactness2} and Remark~\ref{r:1} below for details.
\end{remark}

Once we have the uniform $C^0$-bound of the second fundamental form $\textbf{A}_\ep$, we can derive a uniform bound on the Sobolev norms of $\textbf{A}_\ep$ on a fixed time interval from the evolution equations. Then the convergence of $F_\ep$ and the existence of a solution to the SMCF (\ref{e:SMCF1}) follows from standard arguments.

The rest of the paper is organized as follows. In Section~\ref{s:pre}, we show the Hamiltonian structure and some basic properties of the SMCF. Section~\ref{s:interpolation} is devoted to a compactness theorem for surfaces and the key Theorem~\ref{t:main1}, which is needed in the proof of our main existence result. In Section~\ref{s:app}, we apply the approximating scheme and study the evolution equations of various geometric quantities under the perturbed SMCF~(\ref{e:pSMCF}). Finally the proof of Theorem~\ref{t:main} is finished in Section~\ref{s:existence}.

\vspace{.2in}
\subsection*{Acknowledgements}

The work was initiated during a visit of C.S. at University of Kentucky in 2012, which was supported by the AMS Fan Fund China Exchange program. C.S. would like to thank Prof. Changyou Wang for his generous help and for sharing his ideas. Both authors are grateful to Prof. Youde Wang for his constant support and for pointing out a gap in a previous version of this paper. They would like to thank Prof. Yuxiang Li for sharing his expertise in blow-up analysis and for inspiring discussions on Theorem \ref{t:main1}.

\section{Preliminaries}\label{s:pre}

\subsection{Hamiltonian structure of SMCF}\label{s:sym}

A geometric point of view is to regard the SMCF as a Hamiltonian flow in an infinite dimensional symplectic manifold. In fact, given a Riemannian manifold, the space of co-dimension two submanifolds forms an infinite dimensional symplectic manifold. The induced volume of the
submanifolds defines an energy function on the symplectic manifold and its Hamiltonian flow is exactly the SMCF.

More precisely, suppose $\Si$ is an $n$ dimensional oriented manifold and $(\overline{M},\bar{g})$ is an $(n+2)$ dimensional oriented Riemannian manifold. Let $\mathcal{I}:=\text{Imm}(\Si, \overline{M})/\sim$ denote the
space of smooth immersions from $\Si$ into $\overline{M}$ modulo diffeomorphisms of $\Si$. Obviously, for each
immersion $F\in \text{Imm}(\Si, \overline{M})$, every representative $F\circ \phi$ in its equivalent class $[F]\in\mathcal{I}$ have same image $\overline{\Si}:=F(\Si)\subset \overline{M}$. We denote the normal bundle of $\overline{\Si}$ by $\mathcal{N}\overline{\Si}$ and the space of smooth sections of the normal bundle by $\Ga(\mathcal{N}\overline{\Si})$. Observe that the tangent space of $\mathcal{I}$ at $[F]$ can naturally be identified with $\Ga(\mathcal{N}\overline{\Si})$.

Since $\overline{\Si}$ is a codimension two submanifold, there is a canonical complex structure $J$ on the normal bundle $\mathcal{N}\overline{\Si}$ which simply rotate a normal vector positively by $\pi/2$ in each normal plane. This gives a complex structure on $T_{[F]}\mathcal{I}=\Ga(\mathcal{N}\overline{\Si})$ at each $[F]\in \mathcal{I}$ and yields a global complex structure $\mathcal{J}$ on $\mathcal{I}$.

Given a volume form $d\bar{\mu}$ on $\overline{M}$, we have a natural symplectic structure $\Om$ on $\mathcal{I}$, which was first discovered by Marsden and Weinstein~\cite{MW} for $n=1$, defined pointwisely by
\begin{equation*}
  \Omega|_{[F]}(V, W) = \int_{\overline{\Si}} i_V\circ i_Wd\bar{\mu}|_{\overline{\Si}}
\end{equation*}
for any $V , W\in T_{[F]}\mathcal{I}$. Thus we get an infinite dimensional symplectic manifold $(\mathcal{I},\Om)$.

When there is a Riemannian metric $\bar{g}$ on $\overline{M}$, we may also define an induced metric $G$ on $\mathcal{I}$ by
\begin{equation*}
  G|_{[F]}(V, W) = \int_{\overline{\Si}} \bar{g}(V,W)d\mu,
\end{equation*}
where $d\mu$ is the induced volume form on $\overline{\Si}$.

More importantly, the symplectic structure $\Om$, the metric $G$ and the complex structure $\mathcal{J}$ are compatible, i.e.
\begin{equation*}
  \Om|_{[F]}(V, W) = G|_{[F]}(V, JW).
\end{equation*}

Now the volume of $\overline{\Si}$ defines a functional on $\mathcal{I}$ by
\begin{equation*}
  \mathscr{V}([F]):=\text{vol}(\overline{\Si}).
\end{equation*}
It is well-known that the mean curvature is the gradient vector field of $\mathscr{V}$ in this setting. Using the complex structure $\mathcal{J}$, we can define the corresponding Hamiltonian flow by
\[ \frac{d}{dt}[F]=J\H.\]
which is equivalent to
\[ \(\frac{d}{dt}F\)^\bot=J\H,\]
where $\bot$ denotes the projection to the normal bundle.
Thus the SMCF is just the Hamiltonian flow of the volume function $\mathscr{V}$ in the symplectic manifold $(\mathcal{I}, \Om, \mathcal{J})$.

\subsection{The principal symbol of SMCF}\label{s:symbol}

Although the SMCF and the well-known MCF only differs by the complex structure $J$, the behaviour of SMCF is
totally different from that of MCF. In fact, the SMCF is no longer a (degenerate) parabolic-type equation, but a \sch -type equation since $J$ is skew-symmetric. Here we compute the
principal symbol of SMCF in Euclidean space to illustrate the \sch nature of SMCF. The general case is essentially the same.

For an immersion $F:\Sigma\to \mathbb{R}^{n+2}$, we have
\begin{equation*}
\Delta_{g} F=\textbf{H},
\end{equation*}
where $\Delta_{g}$ is the Laplace operator on $\Sigma$ of the induced metric $g$. By definition, in local coordinates, the induced metric is given by
\begin{equation*}
g_{ij}=\< \frac{\partial F}{\partial x_i}, \frac{\partial F}{\partial x_j} \>.
\end{equation*}
Here, $\< \cdot, \cdot\>$ is the inner product in $\mathbb{R}^{n+2}$. Suppose the standard coordinate on $\mathbb{R}^{n+2}$ is given by $\{y^{\alpha}\}^{n+2}_{\alpha=1}$.  Let $(g^{ij})$
be the inverse matrix of $(g_{ij})$. Then the Christoffel symbol of the induced metric is
\begin{eqnarray*}
\Gamma_{ij}^{k}
&  =  &  \frac{1}{2}g^{kl}\left\{\frac{\partial g_{il}}{\partial x_j}+\frac{\partial g_{jl}}{\partial x_i}
           -\frac{\partial g_{ij}}{\partial x_l}\right\} \nonumber \\
&  =  &  g^{kl}\frac{\partial^2 F^{\beta}}{\partial x_i\partial x_j}\frac{\partial F^{\beta}}{\partial x_l}.
\end{eqnarray*}
Thus we have
\begin{eqnarray}\label{E3.4}
\Delta_{g} F^{\alpha}
&  =  &  g^{ij}\left(\frac{\partial^2 F^{\alpha}}{\partial x_i\partial x_j}-\Gamma_{ij}^{k}\frac{\partial F^{\alpha}}{\partial x_k} \right) \nonumber \\
&  =  &  g^{ij}\left(\frac{\partial^2 F^{\alpha}}{\partial x_i\partial x_j}-
          g^{kl}\frac{\partial^2 F^{\beta}}{\partial x_i\partial x_j}\frac{\partial F^{\beta}}{\partial x_l}\frac{\partial F^{\alpha}}{\partial x_k} \right).
\end{eqnarray}

On the other hand, the complex structure only involves the first order terms of $F$. In fact, by embedding the Grassmannian manifold in the exterior product space $\Ld^n\Real^{n+2}$(see for example \cite{S}), we can write the Gauss map of $F$ by
\[ \rho(F)=\(\frac{\p F}{\p x^1}\wedge \cdots \wedge \frac{\p F}{\p x^n}\)\Big/ \Big|\frac{\p F}{\p x^1}\wedge \cdots \wedge \frac{\p F}{\p x^n}\Big|.\]
Then for any normal vector field $V$ on $F(\Si)$, the action of $J$ can be defined by
\[ JV=*(\rho(F)\wedge V),\]
where $*:\Ld^{n+1}\Real^{n+2}\to \Ld^1\Real^{n+2}=\Real^{n+2}$ is the canonical star operator in $\Real^{n+2}$.

From above discussions, we see that the SMCF (\ref{e:SMCF}) is a quasi-linear system. Denote $P(F)=J\textbf{H}=J\Delta_{g} F$. The linearization operator of $P$ at $F$ is given by
\begin{equation*}
D(P)(F)G =Jg^{ij}\left(\frac{\partial^2 G}{\partial x_i\partial x_j}-
          g^{kl}\<\frac{\partial^2 G}{\partial x_i\partial x_j},\frac{\partial F}{\partial x_l}\>\frac{\partial F}{\partial x_k} \right)+ \text{first\ order\  terms}.
\end{equation*}
The principal symbol is
\begin{eqnarray}\label{E3.6}
\sigma(D(P))(x,\xi)G
&  =  &  Jg^{ij}\left(\xi_i\xi_j G-
          g^{kl}\< G,\frac{\partial F}{\partial x_l}\> \xi_i\xi_j\frac{\partial F}{\partial x_k}\right) \nonumber \\
&  =  & |\xi|^2 J\left(G-g^{kl}\< G,\frac{\partial F}{\partial x_l}\>\frac{\partial F}{\partial x_k}\right)\nonumber \\
&  =  & |\xi|^2 J(G-G^\top)=|\xi|^2 JG^{\perp},
\end{eqnarray}
where $G$ is an any vector in $\mathbb{R}^{n+2}$, $G^\top$ and $G^{\perp}$ are the tangent part and the normal part of $G$ on $\Sigma$. Then we have
\begin{equation*}
\<\sigma(D(P))(x,\xi)G,G \>=\<|\xi|^2 JG^{\perp},G\>=|\xi|^2 \< JG^{\perp},G^{\perp}\>=0.
\end{equation*}
Here we used the fact that $J$ is an isomorphism on the normal bundle. Thus the principal symbol of $P$ is skew-symmetric. In particular, the SMCF is a (degenerate) \sch type
non-linear partial differential equation.

\subsection{Basic properties of SMCF}

In this subsection we show two basic properties of SMCF. Note that these properties hold in arbitrary ambient Riemannian manifold $(\overline{M}, \bar{g})$.

Suppose $F:I\times\Si\to\overline{M}$ is a solution to the SMCF (\ref{e:SMCF}). Denote the inner product induced by $\bar{g}$ by $\<\cdot, \cdot\>$ and the corresponding Levi-Civita connection on $\overline{M}$ by $\overline{\n}$. For each $t\in I$, there is an induced metric $g=g(t)$ and volume form $d\mu = d\mu(t)$ on the surface $\Si$. The most important property of the SMCF is the following lemma.

\begin{lemma}\label{l:metric}
  The induced volume form is preserved under the SMCF. In particular, for a compact manifold, the volume is preserved under the SMCF.
\end{lemma}
\begin{proof}
We prove it point-wisely so that we can take normal coordinates near a point $x\in\Sigma$. The induced metric $g$ is given by
\begin{equation*}
    g_{ij}=\< \frac{\partial F}{\partial x_i}, \frac{\partial F}{\partial x_j}\>.
\end{equation*}
Since $\< \frac{\partial F}{\partial t}, \frac{\partial F}{\partial x_j}\>=0$, it follows that
\begin{equation*}
\begin{aligned}
 \frac{\p}{\p t}g_{ij}&=\< \overline{\n}_i\frac{\partial F}{\partial t}, \frac{\partial F}{\partial x_j}\>
       +\< \frac{\partial F}{\partial x_i}, \overline{\n}_j\frac{\partial F}{\partial t}\> \\
   &= -2\< \frac{\partial F}{\partial t}, \overline{\n}_i\frac{\partial F}{\partial x_j}\>
      =-2\< J\textbf{H}, \textbf{A}(e_i,e_j)\>.
\end{aligned}
\end{equation*}
Consequently, we have
\begin{equation*}
    \frac{\partial}{\partial t}d\mu=\frac{1}{2}g^{kl} \frac{\partial}{\partial t}g_{kl}d\mu
    =-\< J\textbf{H}, \textbf{H}\>d\mu=0.
\end{equation*}
This shows that the volume form $d\mu$, and hence the volume $\text{Vol}(\Si):=\int_\Si d\mu$ is preserved under the SMCF.
\end{proof}

Since the metric on an one dimensional manifold is completely decided by its volume form, we have

\begin{lemma}\label{c1}
 The induced metric is preserved under the 1 dimensional SMCF.
\end{lemma}

Usually along the SMCF, when we talk about the evolution of geometric quantities, e.g. the metric, we first fix a time $t\in I$ and pull back everything induced from the ambient manifold to the base manifold $\Si$. Then we let $t$ vary and derive the equation of a time-dependent quantity on $\Si$.

On the other hand, it is also convenient to consider the whole pull-back bundle $F^*T\overline{M}$ which is defined over the product space $I\times \Si$. This bundle splits in an obvious way into the "spacial" subbundle $\mathcal{H}$ and the normal subbundle $\mathcal{N}$ such that, for each time $t\in I$, the restriction of $\mathcal{H}$ on $\{t\}\times \Si$ is the tangent bundle $F^*T\Si_t$ and the restriction of $\mathcal{N}$ is the normal bundle $F^*N\Si_t$, where $\Si_t=F(t,\Si)$. Moreover, by pulling back the ambient metric $\bar{g}$ and connection $\bar{\n}$ on $\bar{M}$, there are naturally induced metrics $g^\mathcal{H}, g^\mathcal{N}$ and connections $\n^\mathcal{H}, \n^\mathcal{N}$ defined on the bundles $\mathcal{H}$ and $\mathcal{N}$, respectively. It is easy to see that the connections are both compatible with corresponding metrics, i.e.
\[ \n^{\mathcal{H}}g^{\mathcal{H}}=\n^{\mathcal{N}}g^{\mathcal{N}}=0. \]
For a detailed discussion on the structure of bundles of a time-depend immersion, we refer to Chapter 2 of Baker's thesis~\cite{Ba}. In the following, we will simply denote the normal connection $\n^\mathcal{N}$  by $\n$.

Now we regard the complex structure $J$ as a tensor defined on the normal bundle $\mathcal{N}$. The next lemma is crucial for the calculations of evolution equations of SMCF.

\begin{lemma}\label{l:parallel}
The complex structure is parallel w.r.t. the normal connection, i.e. $\nabla J = 0$.
\end{lemma}
\begin{proof}
It suffices to show that for any locally supported unit normal vector field $\textbf{V}$ in the normal bundle, we have
\begin{equation*}
    J\nabla \textbf{V}=\nabla J\textbf{V},
\end{equation*}
where $\nabla$ is the induced connection in the normal bundle. Set $\textbf{W}=J\textbf{V}$ so that $\textbf{V}=-J\textbf{W}$, then $\{\textbf{V},\textbf{W}\}$ forms a local orthonormal
frame. Therefore, for any tangential vector field $X\in T(I\times\Sigma)$, we have
\begin{equation*}
\begin{aligned}
J \nabla_X\textbf{V}   &=  J(\overline{\nabla}_X\textbf{V})^{\perp}\\
    &=J\left(\<\overline{\nabla}_X\textbf{V},\textbf{V}\>\textbf{V}
    +\<\overline{\nabla}_X\textbf{V},\textbf{W}\>\textbf{W}\right)\\
   &= -\<\overline{\nabla}_X\textbf{V},\textbf{W}\>\textbf{V},
\end{aligned}
\end{equation*}
and
\begin{equation*}
\begin{aligned}
 \nabla_X(J\textbf{V})&=\nabla_X\textbf{W}=(\overline{\nabla}_X\textbf{W})^{\perp}\\
    &=\<\overline{\nabla}_X\textbf{W},\textbf{V}\>\textbf{V}
    +\<\overline{\nabla}_X\textbf{W},\textbf{W}\>\textbf{W}\\
    &= -\<\overline{\nabla}_X\textbf{V},\textbf{W}\>\textbf{V}.
\end{aligned}
\end{equation*}
This proves the lemma.
\end{proof}

In particular, Lemma~\ref{l:parallel} shows that the complex structure is parallel along the time direction, i.e. $\n_t J=0$.

\section{Estimate of the Second Fundamental Form}\label{s:interpolation}

\subsection{Estimates for graphs}

In this section, we consider a graph in $(n+m)$-dimensional Euclidean space defined on an $n$-dimensional domain and derive some basic estimates of the second fundamental form. The constants emerging in the calculations may depend on the dimensions, but we will not emphasize it since $n$ and $m$ are always fixed in the application.

Let $u:\Omega\subset \mathbb{R}^n\to \mathbb{R}^m$ be a smooth function, where $\Omega$ is a bounded domain on $\mathbb{R}^n$. Let $\Sigma:=Graph(u)$ denote the graph of $u$, which can be represented by a map $F:\Omega\to \mathbb{R}^{n+m}$ given by
\begin{equation*}\label{E6.1}
    F(x_1,\cdots,x_n):=(x_1,\cdots,x_n,u_1(x_1,\cdots,x_n),\cdots,u_m(x_1,\cdots,x_n)).
\end{equation*}

Here and in the sequel, we will always use $\< \cdot,\cdot\>$ to denote various standard inner products on $\mathbb{R}^{n}$, $\mathbb{R}^m$ and $\mathbb{R}^{n+m}$ without confusions. We will also use $|\cdot|$ and $|\cdot|_g$ to denote the standard Euclidean metric and the induced metric on $\Sigma$, respectively. Moreover, let's agree on the following index ranges
\begin{equation*}
    1\leq i,j,k,l\leq n, \ \ \ 1\leq \alpha,\beta,\gamma\leq m.
\end{equation*}

Denote the partial derivative of $u$ by $D_iu:=D_{x_i}u$.
It is easy to see that a basis of the tangent space $T\Sigma$ of $\Sigma$ can be given by
\begin{equation}\label{E6.2}
    e_i=\frac{\partial F}{\partial x_i}
       =(0,\cdots,0,\underbrace{1}_i,0,\cdots,0,D_iu),
\end{equation}
while a basis of the normal space $N\Sigma$ of $\Sigma$ can be given by
\begin{equation*}
    \nu_{\alpha}
       =(-Du_{\alpha},0,\cdots,0,\underbrace{1}_{n+\alpha},0,\cdots,0).
\end{equation*}
By (\ref{E6.2}), the induced metric on $T\Sigma$ is given by
\begin{equation}\label{E6.4}
    g_{ij}=\< e_i,e_j\>=\delta_{ij}+\< D_iu,D_ju\>.
\end{equation}
Similarly, the induced metric on $N\Sigma$ is given by
\begin{equation}\label{E6.5}
    g_{\alpha\beta}=\< \nu_{\alpha},\nu_{\beta}\>=\delta_{\alpha\beta}+\< Du_{\alpha},Du_{\beta}\>.
\end{equation}
Let $(g^{ij})$ and $(g^{\alpha\beta})$ denote the inverse of $(g_{ij})$ and $(g_{\alpha\beta})$, respectively. Since
\begin{equation*}
    \frac{\partial^2 F}{\partial x_i\partial x_j}=\left(0,\cdots,0,
        \frac{\partial^2 u_1}{\partial x_i\partial x_j},\cdots,\frac{\partial^2 u_m}{\partial x_i\partial x_j}\right)
       =(0,D^2_{ij}u),
\end{equation*}
the second fundamental form of $\Sigma$ is given by
\begin{equation*}
    \textbf{A}(e_i,e_j)=\left(\frac{\partial^2 F}{\partial x_i\partial x_j}\right)^{\perp}
    =g^{\alpha\beta}\<\frac{\partial^2 F}{\partial x_i\partial x_j},\nu_{\beta}\>\nu_{\alpha}
    =g^{\gamma\beta}\frac{\partial^2 u_{\beta}}{\partial x_i\partial x_j}\nu_{\gamma},
\end{equation*}
and the component of the second fundamental form is
\begin{equation}\label{E6.7}
    h_{\alpha ij}=\<\textbf{A}(e_i,e_j),\nu_{\alpha}\>=\frac{\partial^2 u_{\alpha}}{\partial x_i\partial x_j}.
\end{equation}
From (\ref{E6.4}), we can easily see that the eigenvalues $\{\lambda_i\}_{1\leq i\leq n}$ of $(g_{ij})$ satisfy
\begin{equation}\label{E6.8}
    1\leq \lambda_i\leq 1+|Du|^2.
\end{equation}
Similarly,
the eigenvalues $\{\mu_{\alpha}\}_{1\leq \alpha\leq m}$ of $(g_{\alpha\beta})$ satisfy
\begin{equation*}
    1\leq \mu_{\alpha}\leq 1+|Du|^2.
\end{equation*}
Therefore, the eigenvalues of $(g^{ij})$ and $(g^{\alpha\beta})$ can be bounded by
\begin{equation}\label{E6.10}
    \frac{1}{1+|Du|^2}\leq \lambda_i^{-1}\leq 1, \ \ \ \frac{1}{1+|Du|^2}\leq \mu_{\alpha}^{-1}\leq 1.
\end{equation}

Since
\begin{equation*}
    |\textbf{A}|_g^2=g^{ik}g^{jl}g^{\alpha\beta}h_{\alpha ij}h_{\beta kl}
       =g^{ik}g^{jl}g^{\alpha\beta}\frac{\partial^2 u_{\alpha}}{\partial x_i\partial x_j}\frac{\partial^2 u_{\beta}}{\partial x_k\partial x_l},
\end{equation*}
it follows easily from  (\ref{E6.10}) that

\begin{lemma}\label{lem6.1}
\begin{equation}\label{E6.11}
    |\textbf{A}|_g^2\leq |D^2u|^2\leq (1+|Du|^2)^3|\textbf{A}|_g^2.
\end{equation}
\end{lemma}

Next, in order to estimate the derivatives of $\textbf{A}$, we need to compute the Christoffel symbols associated to the connection. Denote the Levi-Civita connection on $\mathbb{R}^{n+m}$ by $\overline\nabla$ and the induced connection on $\Sigma$ by $\nabla$ respectively. The induced connection $\nabla$ applies to $\textbf{A}$ and naturally extends to tensor fields in $N\Sigma\otimes (T^*\Sigma)^k, k\in \mathbb{N}$. By definition, the Christoffel symbols are given
by
\begin{equation*}\label{E6.12}
    \nabla_{e_i}e_j=(\overline\nabla_{e_i}e_j)^T=\Gamma_{ij}^k e_k,\ \ \
    \nabla_{e_i}\nu_{\alpha}=(\overline\nabla_{e_i}\nu_{\alpha})^{\perp}=\Gamma_{i\alpha}^{\beta} \nu_{\beta}.
\end{equation*}
Since
\begin{equation*}
    (\overline\nabla_{e_i}e_j)^T= \left(\frac{\partial^2 F}{\partial x_i\partial x_j}\right)^{T}
        =g^{kl}\<\frac{\partial^2 F}{\partial x_i\partial x_j},e_l\> e_k=g^{kl}\< D_{ij}u,D_l u\> e_k,
\end{equation*}
we have
\begin{equation}\label{E6.13}
    \Gamma_{ij}^k=g^{kl}\< D_{ij}u,D_l u\>=g^{kl}\frac{\partial^2 u_{\alpha}}{\partial x_i\partial x_j}\frac{\partial u_{\alpha}}{\partial x_l}.
\end{equation}
Similarly, since
\begin{eqnarray*}
 (\overline\nabla_{e_i}\nu_{\alpha})^{\perp}
   &=& \left(\frac{\partial}{\partial x_i}\nu_{\alpha}\right)^{\perp}
       = \left(-\frac{\partial^2 u_{\alpha}}{\partial x_i\partial x_1},\cdots,-\frac{\partial^2 u_{\alpha}}{\partial x_i\partial x_n},
             0,\cdots,0\right)^{\perp}\\
   &=& g^{\beta\gamma}\<\left(-\frac{\partial^2 u_{\alpha}}{\partial x_i\partial x_1},\cdots,-\frac{\partial^2 u_{\alpha}}{\partial x_i\partial x_n},
             0,\cdots,0\right),\nu_{\gamma} \> \nu_{\beta}\\
   &=& g^{\beta\gamma}\frac{\partial^2 u_{\alpha}}{\partial x_i\partial x_k}\frac{\partial u_{\gamma}}{\partial x_k}\nu_{\beta},
\end{eqnarray*}
we have
\begin{equation}\label{E6.14}
    \Gamma_{i\alpha}^{\beta}=g^{\beta\gamma}\frac{\partial^2 u_{\alpha}}{\partial x_i\partial x_k}\frac{\partial u_{\gamma}}{\partial x_k}.
\end{equation}

By fixing the chosen frame, we may regard $\Ga$ as a vector field in some Euclidean space with components given by all $\Ga_{ij}^k$ and $\Ga_{i\al}^\be$. Or equivalently, we may introduce a standard metric such that the basis $\{e_i\}_{i=1}^n$ and $\{\nu_\al\}_{\al=1}^2$ are orthonormal. Similarly, we can also treat $\textbf{A}, u, g$ and
their derivatives as vectors in (probably different dimensional) Euclidean spaces. Then we still use $|\cdot|$ to denote their norms w.r.t. the standard metric in corresponding Euclidean spaces.
Moreover, we can simply write (\ref{E6.7}) as $\textbf{A}=D^2u$ and rewrite (\ref{E6.13}) and (\ref{E6.14}) as
\begin{equation}\label{E6.141}
  \Gamma=D^2u*Du*g^{-1},
\end{equation}
where $*$ denotes multiple linear combinations of components of the vectors. It follows from (\ref{E6.10}) that
\begin{equation}\label{E6.142}
  |\Gamma|\le C|D^2u|\cdot|Du|.
\end{equation}

\begin{lemma}\label{lem6.2}
There exists a constant $C$ such that
\begin{equation*}
    |\nabla\textbf{A}|_g\leq |D^3u|+C|Du||D^2u|^2,
\end{equation*}
and
\begin{equation*}
    |D^3u|\leq (1+|Du|^2)^2|\nabla\textbf{A}|_g+C|Du||D^2u|^2.
\end{equation*}
\end{lemma}

\begin{proof}
By definition, we have
\begin{equation}\label{E6.6}
|\nabla\textbf{A}|_g^2=h_{\alpha ij,k}h_{\beta pq,l}g^{ip}g^{jq}g^{kl}g^{\alpha\beta},
\end{equation}
where $h_{\alpha ij,k}$ stands for the $k$-th covariant derivative of $h_{\alpha ij}$, i.e.
\begin{equation*}\label{E6.17}
    h_{\alpha ij,k}=\frac{\partial h_{\alpha ij}}{\partial x_k}+\Gamma_{k i}^lh_{\alpha lj}+\Gamma_{k j}^lh_{\alpha il}+\Gamma_{k\alpha}^{\beta}h_{\beta ij}.
\end{equation*}
Using our convention, we may simply write
\begin{equation}\label{E6.9}
  \n \textbf{A} = D\textbf{A} + \Ga*\textbf{A}=D^3u + \Ga*D^2u.
\end{equation}

By (\ref{E6.6}) and (\ref{E6.10}), we have
\begin{equation*}
    |\nabla\textbf{A}|_g\leq |\n\textbf{A}|\le|D^3u|+C|\Ga|\cdot|D^2u|,
\end{equation*}
and
\begin{equation*}
    |\nabla\textbf{A}|_g\geq \frac{1}{(1+|Du|^2)^2}|\n\textbf{A}|
    \ge\frac{1}{(1+|Du|^2)^2}(|D^3u|-C|\Ga|\cdot|D^2u|).
\end{equation*}
Then the lemma follows from (\ref{E6.142}) and Lemma~\ref{lem6.1} easily.
\end{proof}

In order to compute higher order derivatives, we first note that from (\ref{E6.4}),
\begin{equation*}
    \frac{\partial g_{ij}}{\partial x_k}=\frac{\partial^2 u_{\alpha}}{\partial x_k\partial x_i}\frac{\partial u_{\alpha}}{\partial x_j}
        +\frac{\partial^2 u_{\alpha}}{\partial x_k\partial x_j}\frac{\partial u_{\alpha}}{\partial x_i},
\end{equation*}
which implies
\begin{equation*}
    \frac{\partial g^{ij}}{\partial x_k}=-g^{ip}g^{jq}\left(\frac{\partial^2 u_{\alpha}}{\partial x_k\partial x_p}\frac{\partial u_{\alpha}}{\partial x_q}
        +\frac{\partial^2 u_{\alpha}}{\partial x_k\partial x_q}\frac{\partial u_{\alpha}}{\partial x_p}\right).
\end{equation*}
Similarly, from (\ref{E6.5}), we have
\begin{equation*}
    \frac{\partial g_{\alpha\beta}}{\partial x_k}=\frac{\partial^2 u_{\alpha}}{\partial x_k\partial x_i}\frac{\partial u_{\beta}}{\partial x_i}
        +\frac{\partial^2 u_{\beta}}{\partial x_k\partial x_i}\frac{\partial u_{\alpha}}{\partial x_i},
\end{equation*}
which implies
\begin{equation*}
    \frac{\partial g^{\alpha\beta}}{\partial x_k}
    =-g^{\alpha\gamma}g^{\beta\delta}\left(\frac{\partial^2 u_{\gamma}}{\partial x_k\partial x_i}\frac{\partial u_{\delta}}{\partial x_i}
        +\frac{\partial^2 u_{\delta}}{\partial x_k\partial x_i}\frac{\partial u_{\gamma}}{\partial x_i}\right).
\end{equation*}
Equivalently, we may write
\begin{equation}\label{E6.3}
  Dg^{-1}=D^2u*Du*g^{-1}*g^{-1}.
\end{equation}

\begin{lemma}\label{lem6.3}
There exists a polynomial $P_k$ depending on $k$, such that
\begin{equation}\label{E6.20}
    |\nabla^k\textbf{A}|_g\leq |D^{k+2}u|+P_k(|Du|)\sum |D^{j_1+1}u|\cdots  |D^{j_s+1}u|
\end{equation}
and
\begin{equation}\label{E6.21}
    |D^{k+2}u|\leq (1+|Du|^2)^\frac{k+3}{2}|\nabla^k\textbf{A}|_g+P_k(|Du|)\sum |D^{j_1+1}u|\cdots  |D^{j_s+1}u|,
\end{equation}
where and the summations are taken over all indices $(j_1,\cdots,j_s)$ satisfying
\begin{equation}\label{E6.22}
j_1\geq j_2\cdots \geq j_s, \ k\geq j_i \geq 1, \ j_1+j_2+\cdots j_s=k+1.
\end{equation}
\end{lemma}
\begin{proof}
We prove the lemma by induction. The case $k=0$ and $k=1$ has already been proved in Lemma \ref{lem6.1} and Lemma \ref{lem6.2} respectively. For $k\ge 2$, first note that
\begin{equation*}\label{E6.23}
|\nabla^k\textbf{A}|_g^2=g^{\alpha\beta}g^{p_1q_1}g^{p_2q_2}g^{i_1j_1}\cdots g^{i_kj_k}
h_{\alpha p_1p_2,i_1\cdots i_k}h_{\beta q_1q_2,j_1\cdots j_k}.
\end{equation*}
From (\ref{E6.10}), we see that
\begin{equation}\label{E6.24}
|\nabla^k\textbf{A}|_g\leq |\nabla^k\textbf{A}| \leq (1+|Du|^2)^{\frac{k+3}{2}}|\nabla^k\textbf{A}|_g.
\end{equation}
Since
\begin{equation}\label{E6.25}
  \n^k\textbf{A}=D\n^{k-1}\textbf{A} + \Ga*\n^{k-1}\textbf{A}.
\end{equation}
Using (\ref{E6.9}) and (\ref{E6.25}), we can verify by induction that
\begin{equation*}
\n^k\textbf{A}=D^{k+2}u+\tilde{P}_k(g^{-1}, Du, D^2u, \cdots, D^{k+1}u).
\end{equation*}
where $\tilde{P}_k$ are multiple linear form given by
$$\tilde{P}_k=\underbrace{g^{-1}*\cdots g^{-1}}_k*\underbrace{Du*\cdots Du}_k*\sum D^{j_1+1}u*\cdots* D^{j_s+1}u$$
with the summation taken over indices satisfying (\ref{E6.22}).
Therefore, there exists a polynomial $P_k$ depending only on $k$, such that
\begin{equation*}
|\tilde{P}_k|\leq P_k(|Du|)\sum |D^{j_1+1}u|\cdots|D^{j_s+1}u|.
\end{equation*}
Now (\ref{E6.20}) and (\ref{E6.21}) follows from (\ref{E6.24}).
\end{proof}

Next we suppose that the function $u$ is defined on a disk $D_r\subset \Real^n$ centered at the origin with radius $r>0$. By (\ref{E6.7}), we may identify
$\textbf{A}$ with $D^2u$. For any non-negative integer $k$ and positive number $p$, there is a usual Sobolev norm of the Hessian $D^2u$ given by
\begin{equation*}
\norm{D^2u}_{W^{k,p}} = \left(\int_{D_r} \sum_{l=0}^k \abs{D^l D^2u}^p dx\right)^{\frac{1}{p}}.
\end{equation*}
On the other hand, we can define a Sobolev-type norm of $\textbf{A}$ by
\begin{equation}\label{e:sobolev-norm}
\norm{\textbf{A}}_{H^{k,p}} = \left(\int_\Si \sum_{l=0}^k \abs{\n^l \textbf{A}}_g^p d\mu\right)^{\frac{1}{p}}.
 \end{equation}
In particular, we have $\norm{\textbf{A}}_{L^p}:=\norm{\textbf{A}}_{H^{0,p}}$.

The next lemma shows that if $|Du|$ is bounded, then the Sobolev norms of $\textbf{A}$ can be bounded by the usual Sobolev norms of the Hessian $D^2u$. The proof of the lemma follows closely
that of Lemma 2.2 in \cite{DW2}.

\begin{lemma}\label{compare-I}
Let $r>0, \al>0$ be positive numbers. Suppose $\Si$ is a smooth graph represented by $u:D_{r}\subset \Real^n\to\Real^m$ with $|Du|\leq\al$, then for any $k\ge 0$,
\begin{equation}\label{e.compare-I}
\norm{\textbf{A}}_{H^{k,2}}\leq C\sum_{s=1}^{k+1}\norm{D^2u}^s_{W^{k,2}},
\end{equation}
where the constant $C$ depends on $r$, $\al$ and $k$.
\end{lemma}

\begin{proof}
For convenience, set $\sigma:=D^2u$. Then the inequality (\ref{E6.20}) can be written as
\begin{equation*}
|\nabla^k\textbf{A}|_g\leq |D^{k}\si|+P_k(|Du|)\sum |D^{j_1-1}\si|\cdots  |D^{j_s-1}\si|,
\end{equation*}
where the summation is taken over all indices satisfying (\ref{E6.22}).
Since $|Du|\le \al$, we may integrate to get
\begin{equation}\label{equ1}
\norm{\n^k\textbf{A}}_{L^2}\leq \norm{D^{k}\si}_{L^2} + C\sum \Norm{|D^{j_1-1}\si|\cdots  |D^{j_s-1}\si|}_{L^2},
\end{equation}
where $C$ depends on $k$ and $\al$.
For the second term in the last inequality, we may apply H\"older's inequality to get
\begin{equation}\label{equ2}
 \Norm{|D^{j_1-1}\sigma|\cdots  |D^{j_s-1}\sigma|}_{L^{2}}\leq \norm{D^{j_1-1}\sigma}_{L^{q_1}}\cdots \norm{D^{j_s-1}\sigma}_{L^{q_s}},
\end{equation}
where the numbers $q_1, \cdots q_s$ satisfies
\begin{equation*}
\frac{1}{q_1}+\cdots+\frac{1}{q_s}=\frac12.
\end{equation*}

We claim that the numbers $q_i$ can be chosen such that there exists $\frac{j_i-1}{k}\le a_i\le 1$ satisfying the equality
\begin{equation}\label{E7.15}
\frac{1}{q_i}=\frac{j_i-1}{n}+a_i(\frac{1}{2}-\frac{k}{n})+(1-a_i)\frac12.
\end{equation}
If the claim is true, then we can apply the Gagliardo-Nirenberg interpolation inequality(see for example \cite{Friedman}, page 27, Theorem 10.1) to $\si$ to get
\begin{equation}\label{equ3}
\norm{D^{j_i-1}\sigma}_{L^{q_i}}\leq C\norm{\sigma}^{a_i}_{W^{k,2}}\norm{\sigma}^{1-a_i}_{L^2}\le C\norm{\sigma}_{W^{k,2}}
\end{equation}
where the constant $C$ depends on $r$, $k$, $j_i$ and $q_i$.
Combining (\ref{equ1}), (\ref{equ2}) and (\ref{equ3}), we see that (\ref{e.compare-I}) follows easily.

The proof of the above claim is a direct calculation and we refer to \cite{DW2}, page 1452.
\end{proof}

\subsection{Compactness results for surfaces}\label{s:compactness}

Now we restrict ourselves to the case of codimension two surfaces in $\Real^4$, i.e. the case $n=m=2$. We first show that in this case, the inverse of Lemma~\ref{compare-I} is also
correct. Namely, if in addition $|D^2u|$ is bounded, then the Sobolev norms of $\textbf{A}$ and $D^2 u$ are equivalent.

\begin{lemma}\label{compare-II}
Let $r>0, \al>0, \beta>0$ be positive numbers. Suppose $\Si$ is a smooth graph represented by $u:D_{r}\subset \Real^2\to\Real^m$ with $|Du|\leq\al$ and $|D^2u|\leq \beta$, then for any
$k\ge 0$,
\begin{equation}\label{e.compare-II}
\norm{D^2u}_{W^{k,2}}\leq C\sum_{s=1}^{k}\norm{\textbf{A}}^s_{H^{k,2}},
\end{equation}
where the constant $C$ depends on $r$, $\al$, $\beta$ and $k$.
\end{lemma}
\begin{proof}
We will prove (\ref{e.compare-II}) by induction. As before,  we set $\sigma:=D^2u$.

The case $k=0$ follows directly from (\ref{E6.8}) and (\ref{E6.11}), i.e.,
\begin{equation}\label{equ4}
\norm{\si}_{L^{2}}\leq C_0\norm{\textbf{A}}_{L^{2}}.
\end{equation}
By our assumption and Lemma \ref{lem6.2}, we have
\begin{equation*}
|D\si|\leq C(|\nabla \textbf{A}|_g+|\textbf{A}|_g).
\end{equation*}
Thus the case $k=1$ also holds true.

Now assume by induction that (\ref{e.compare-II}) holds for any $k\geq 1$. To prove the lemma, it suffices to estimate $\norm{D^{k+1}\si}_{L^2}$ in terms of $\norm{\textbf{A}}_{H^{k+1,2}}$.

Recall that by (\ref{E6.21}), we have
\begin{equation}\label{equ11}
    \norm{D^{k+1}\si}_{L^2}\leq C\norm{\nabla^{k+1}\textbf{A}}_{L^2}+C\sum \Norm{|D^{j_1-1}\si|\cdots  |D^{j_s-1}\si|}_{L^2},
\end{equation}
where the summations are taken over all indices $(j_1,\cdots,j_s)$ satisfying
\begin{equation}\label{equ5}
j_1\geq j_2\cdots \geq j_s, \ k+1\geq j_i \geq 1, \ j_1+j_2+\cdots j_s=k+2.
\end{equation}
To estimate the second term in the right hand side of (\ref{equ11}), we consider two cases:

\emph{Case 1: $j_1=k+1$}: In this case, by (\ref{equ5}), it obvious that $s=2$ and $j_2=1$. Then the term is simply bounded by
\begin{equation*}
 \Norm{|D^k\sigma|\cdot |\sigma|}_{L^{2}}\leq \norm{\si}_{L^\infty}\norm{D^{k}\sigma}_{L^{2}}\leq \be C\sum_{s=1}^{\gamma_k}\norm{\textbf{A}}^s_{H^{k,2}},
\end{equation*}
where the last inequality used the induction assumption.

\emph{Case 2: $j_1<k+1$}: In this case, by (\ref{equ5}), we have $1\leq j_i\leq k$ for any $1\leq i\leq s$. Applying H\"older's inequality, we get
\begin{equation}\label{equ6}
 \Norm{|D^{j_1-1}\sigma|\cdots  |D^{j_s-1}\sigma|}_{L^{2}}\leq \norm{D^{j_1-1}\sigma}_{L^{q_1}}\cdots \norm{D^{j_s-1}\sigma}_{L^{q_s}},
\end{equation}
where the numbers $q_1, \cdots q_s\in[2,\infty]$ satisfies
\begin{equation}\label{equ7}
\frac{1}{q_1}+\cdots+\frac{1}{q_s}=\frac12.
\end{equation}

Then we can find number $a_i$ decided by the following equality
\begin{equation}\label{equ9}
\frac{1}{q_i}=\frac{j_i-1}{2}+a_i(\frac{1}{2}-\frac{k}{2})+(1-a_i)\frac12=\frac{j_i-ka_i}{2}.
\end{equation}
Since $2\le q_i\le +\infty$, one can easily verify that $\frac{j_i-1}{k}\le a_i\le 1$. Therefore, we may apply the Gagliardo-Nirenberg interpolation inequality to get
\begin{equation}\label{equ10}
\norm{D^{j_i-1}\sigma}_{L^{q_i}}\leq C\norm{\sigma}^{a_i}_{W^{k,2}}\norm{\sigma}^{1-a_i}_{L^2}
\end{equation}
where the constant $C$ depends on $r$, $k$, $j_i$ and $q_i$. Putting (\ref{equ10}) into (\ref{equ6}) yields
\begin{equation*}
 \Norm{|D^{j_1-1}\sigma|\cdots  |D^{j_s-1}\sigma|}_{L^{2}}\leq C\norm{\sigma}^{\sum_{i}a_i}_{W^{k,2}}\norm{\sigma}^{\sum_{i}(1-a_i)}_{L^2}.
 \end{equation*}
By (\ref{equ7}) and (\ref{equ9}), it is easy to see that
\begin{equation*}
\sum_i a_i=1+\frac{1}{k}, \ \ \ \sum_{i}(1-a_i)=s-1-\frac{1}{k}.
\end{equation*}
Since $2\leq s\leq k+2$, it follows
\begin{equation}\label{equ12}
 \sum\Norm{|D^{j_1-1}\sigma|\cdots  |D^{j_s-1}\sigma|}_{L^{2}}\leq C\norm{\sigma}^{1+\frac{1}{k}}_{W^{k,2}}\norm{\sigma}^{s-1-\frac{1}{k}}_{L^2}.
 \end{equation}
Combining (\ref{equ11}) and (\ref{equ12}), and noting that
\begin{equation*}
\norm{\sigma}_{L^2}\leq\norm{\sigma}_{L^\infty}|D_r|^{\frac12}\le  \beta \sqrt{\pi}r,
\end{equation*}
we get
\begin{equation*}
    \norm{D^{k+1}\si}_{L^2}\leq C\norm{\nabla^{k+1}\textbf{A}}_{L^2}+C\norm{\sigma}^{1+\frac{1}{k}}_{W^{k,2}}.
\end{equation*}
Using the induction assumption, we conclude that (\ref{e.compare-II}) holds for $k+1$ and the lemma follows.
\end{proof}

Then following Langer~\cite{Langer}, we can prove the following compactness theorem.
\begin{theorem}\label{t:compactness1}
Given a compact two dimensional surface $\Si$, an integer $k\ge 1$ and constants $\be, \mathcal{A},\mathcal{V} > 0$, let $\mathcal{M}$ be the set of immersions $F:\Si\to \Real^4$
satisfying $\norm{\textbf{A}}_{L^\infty}\le \be$, $\norm{\textbf{A}}_{H^{k,2}}\le \mathcal{A}$, $\text{Vol}(\Si)\le \mathcal{V}$ and $0\in F(\Si)$. Then for any sequence $F_i$ in $\mathcal{M}$,
there exists a sequence of diffeomorphisms $\phi_i$ on $\Si$, such that $F_i\circ \phi_i$ sub-converges in $W^{k+2,2}$ weakly and $C^{k,\al}$ strongly to an immersion $F_\infty\in
\mathcal{M}$, where $0< \al<1$.
\end{theorem}

There is also a localized version of the above theorem which is useful in blow-up analysis.

\begin{theorem}\label{t:compactness2}
Given a compact two dimensional surface $\Si$, an integer $k\ge 1$ and constants $\be, \mathcal{A}, \mathcal{V}(R)> 0$ where $\mathcal{V}(R)$ depends on $R$, let $F_i:\Si\to \Real^4$ be a
sequence of immersions satisfying $\norm{\textbf{A}(F_i)}_{L^\infty}\le \be$, $\norm{\textbf{A}(F_i)}_{H^{k,2}}\le \mathcal{A}$, $0\in F_i(\Si)$ and
\[\text{Vol}(\Si_i(R))\le \mathcal{V}(R)\]
where $\Si_i(R) = \Si_i\cap B(R)$ is the portion of the immersed surface $\Si_i:=F_i(\Si)$ bounded in the Euclidean ball of radius $R$. Then there exists a surface $\tilde{\Si}$ without
boundary, an immersion $F_\infty:\tilde{\Si}\to \Real^4$ and  a sequence of diffeomorphisms $\phi_i$, such that $F_i\circ \phi_i$ sub-converges to $F$ on any compact subset of
$\tilde{\Si}$ in $W^{k+2,2}$ weakly and $C^{k,\al}$ strongly, where $0< \al<1$. Here $\phi_i: U_i\to F_i^{-1}(\Si_i(R))$ are defined on open sets $U_i\subset \tilde{\Si}$ where $U_i\subset
\subset U_{i+1}$ and $\tilde{\Si}=\cup_{i=1}^{\infty}U_i$.
\end{theorem}

For simplicity, we say that $F_i$ converges to $F_\infty$ weakly in $W^{k+2,2}$-topology and strongly in $C^{k,\al}$-topology to $F_\infty$ in Theorem~\ref{t:compactness1}, and $F_i$
converges locally to $F_\infty$ in the same topology in Theorem~\ref{t:compactness2}, respectively.

To prove the above theorems, we follow the idea of Langer~\cite{Langer} and give an outline. Since the second fundamental forms has uniform upper bounds, there exists a uniform pair of number $r>0$ and $\al>0$, such that each $\Si_i$ is a $(r, \al)$-immersion. Namely, for any point $y\in \Si_i$,
there is a neighborhood of $y$ which can be represented by a graph $u:D_r\to \Real^2$ such that $|Du|<\al$. Then the surfaces can be each represented by a graph system where the graphs is defined on the disc $D_r$. Moreover, the number of graphs in each system is finite since we have uniform bounds on the volume. On each disk $D_r$, we can apply Lemma~\ref{compare-II} to find that $u$ has
uniformly bounded $W^{k+2,2}$-norms in terms of $\norm{\textbf{A}}_{H^{k,2}}$. Therefore, the graphs converge weakly in $W^{k+2,2}(D_r)$ and strongly in $C^{k,\al}(D_r)$. It follows that the graph systems converges by passing to subsequences. Finally, we may construct a sequence of diffeomorphisms on the surface such that the immersions composed with the diffeomorphisms converges in the desired spaces. Here we omit the details and refer the readers to~\cite{Langer} and Breuning's paper~\cite{Breu}.

\subsection{Uniform estimate for second fundamental form}\label{s:uniform-estimate}

Recall that in~\cite{DW2}, Ding and Wang generalized the classical Gagliardo-Nirenberg interpolation inequality to sections of vector bundles. In particular, if we regard the second
fundamental form $\textbf{A}$ as a section of the bundle $T^*\Si\otimes T^*\Si\otimes N\Si$, then it follows from \cite{DW2} that an interpolation inequality holds for $\textbf{A}$. However, if
the metric of the underlying manifold is varying, which is the case of SMCF, the Sobolev constant will vary.

Here we use blow up techniques to establish a uniform embedding theorem for $\textbf{A}$, i.e. Theorem~\ref{t:main1}, which will play a crucial role in the proof of our main theorem~\ref{t:main}. The blow up techniques applied here is analogous to the one used in the study of Willmore flow, see for example~\cite{KS2,KL}.
From now on, we simply denote the induced volume of $\Sigma$ by $|\Sigma|:=\text{Vol}(\Si)$. First we quote the following version of Simon's inequality (see (1.3)
in~\cite{Simon1} or Lemma 4.1 in~\cite{KS2}).

\begin{lemma}\label{l:simon}
Suppose $F:\Si\to \Real^n$ is a compact immersed surface. Then for any $0<R<\infty$ and $\Si(R)=F(\Si)\cap B(R)$, one has
\begin{equation*}
  \frac{|\Si(R)|}{R^2}\le C\int_\Si |\textbf{H}|^2d\mu.
\end{equation*}
\end{lemma}

Now we can prove the key estimate (Theorem~\ref{t:main1}).

\begin{theorem}\label{t:sobolev}
Given positive numbers $B$ and $m$, there exists a constant $C(B,m)$ such that for any immersed compact surface $\Si^2\subset\Real^4$ satisfying
\begin{equation*}
\norm{\textbf{A}}_{H^{2,2}}\le B \text{~~and~~} |\Si|\ge m,
\end{equation*}
there holds
\begin{equation*}
\norm{\textbf{A}}_{C^0}\le C(B,m).
\end{equation*}
\end{theorem}
\begin{proof}
We argue by contradiction and apply a blowing-up technique following Section 4 of \cite{KS2}. Suppose the theorem is false. Then there exists a sequence of compact surfaces $\Si_k\subset
\Real^4$ with $\norm{\textbf{A}_k}_{W^{2,2}}\le B$ and $|\Si_k|\ge m$ such that $\lim_{k\to \infty}\norm{\textbf{A}_k}_{C^0}=+\infty$. We are going to show that this is impossible by using
blow-up analysis.

Since $\Si_k$ is compact, $\norm{\textbf{A}_k}_{C^0}$ is attained at some point $y_k\in\Si_k$ such that
\[|\textbf{A}_k(y_k)|=\max_{y\in \Si_k}|\textbf{A}_k(y)|.\]
Denote $r_k:=1/|\textbf{A}_k(y_k)|$. Then $\lim_{k\to \infty}r_k=0$.
Thus we can define a sequence of rescaled surfaces $\Si_k'=\frac{\Si_k-y_k}{r_k}$. Denote the corresponding second fundamental form by $\textbf{A}_k'$. Let $g_k$ and $g_k'$ be the induced
metric on $\Si_k$ and $\Si_k'$ respectively. By rescaling properties, we have
\begin{equation}\label{eq7}
|\textbf{A}_k'|_{g_k'}=r_k|\textbf{A}_k|_{g_k}\le r_k|\textbf{A}_k(y_k)|_{g_k}=1.
\end{equation}
Moreover,
\begin{equation}\label{eq1}
\norm{\textbf{A}_k'}_{L^2,g_k'}=\norm{\textbf{A}_k}_{L^2,g_k}
\end{equation}
and
\begin{equation}\label{eq2}
\norm{\n_{g_k'}\textbf{A}_k'}_{L^2,g_k'}=r_k\norm{\n_{g_k}\textbf{A}_k}_{L^2,g_k}.
\end{equation}

Then one can verify that the sequence of rescaled surfaces $\Si_k'$ satisfies all the requirements of Theorem~\ref{t:compactness2}. In particular, the local volume bound follows from Lemma~\ref{l:simon}. Consequently there exists a subsequence of the surfaces, which we still denote by $\Si_k'$, such that $\Si_k'$ converges locally to a complete surface $\Si_0$
weakly in $W^{4,2}$ and strongly in $C^{2,\al}$.

Let $\textbf{A}_0$ and $g_0$ denote the second fundamental form and induced metric of the limit surface $\Si_0$ respectively. Note that $\textbf{A}_0\in C^\al$ is continuous. It follows from (\ref{eq7}) that
\begin{equation}\label{eq3}
\norm{\textbf{A}_0}_{C^0,g_0} = \lim_{k\to \infty} r_k|\textbf{A}_k(y_k)|_{g_k} = 1.
\end{equation}
Moreover, by (\ref{eq1}) and (\ref{eq2}), we have
\begin{equation}\label{eq9}
\norm{\textbf{A}_0}_{L^2,g_0}=\lim_{k\to\infty}\norm{\textbf{A}_k}_{L^2,g_k}\le B
\end{equation}
and
\begin{equation}\label{eq6}
\norm{\n_{g_0}\textbf{A}_0}_{L^2,g_0}=\lim_{k\to\infty}r_k\norm{\n_{g_k}\textbf{A}_k}_{L^2,g_k}\le \lim_{k\to \infty}r_kB=0.
\end{equation}
Note that by Kato's inequality, we have
\[ |\n_{g_0}|\textbf{A}_0|_{g_0}|\le |\n_{g_0}\textbf{A}_0|_{g_0}, \]
which together with (\ref{eq6}) yields
\begin{equation*}
\int_{\Si_0}|\n_{g_0}|\textbf{A}_0|_{g_0}|^2 d\mu_{g_0}\le \norm{\n_{g_0}\textbf{A}_0}_{L^2,g_0}^2=0.
\end{equation*}
Thus we find $\n_{g_0}|\textbf{A}_0|_{g_0}=0$ a.e. on $\Si_0$, which implies that $|\textbf{A}_0|_{g_0}$ is constant since $\Si_0$ is connnected. It follows from (\ref{eq3}) that $|\textbf{A}_0|_{g_0}\equiv 1$.

Now, if $\Si_\infty$ is compact, then it has finite volume. Since in this case $\Si_\infty$ is the only component of the limit surface, this contradicts with the assumption of finite lower bound of the volume of $\Si_k$. Otherwise $\Si_\infty$
is complete and non-compact, but its mean curvature is bounded in view of $|\textbf{A}_0|_{g_0}\equiv 1$. Thus $\Si_\infty$ has infinite volume (see for example \cite{CL}), which contradicts with the finiteness of $\norm{\textbf{A}_0}_{L^2,g_0}$.
\end{proof}

\begin{remark}\label{r:1}
Obviously, using Theorem~\ref{t:sobolev}, we can replace the requirement of upper bound on $\norm{\textbf{A}}_{L^\infty}$ with a lower bound on the volume of the surfaces, then the
compactness results in Theorem~\ref{t:compactness1} and Theorem~\ref{t:compactness2} still holds.
\end{remark}

\section{Short Time Existence of SMCF}\label{s:app}

\subsection{The perturbed flow}

One goal in this section is to prove the local existence of two dimensional SMCF in Euclidean space $\Real^4$. But for the moment, let us consider a general approximating scheme for $n$ dimensional SMCF in $\Real^{n+2}$.

To obtain a local solution to the SMCF (\ref{e:SMCF1}), we will consider the perturbed SMCF (\ref{e:pSMCF})
\begin{equation}\label{app-flow}
\left\{\begin{aligned}
         &\frac{\partial F}{\partial t}=J\textbf{H}+\varepsilon\textbf{H}, \\
         &F(0,\cdot)=F_0,
    \end{aligned}\right.
\end{equation}
where $\varepsilon>0$ is a positive number. The idea is to solve the perturbed SMCF (\ref{e:pSMCF}) and approach the original SMCF (\ref{e:SMCF1}) by letting $\varepsilon$ go to zero.

Similar to the argument in Section~\ref{s:symbol}, it is easy to check that the system (\ref{app-flow}) is a degenerate parabolic system. In fact, if we set
$P_{\varepsilon}(F)=J\textbf{H}+\varepsilon\textbf{H}=J\Delta_{\Sigma}F+\varepsilon\Delta_{\Sigma}F$, then the principal symbol of $P_{\varepsilon}$ is
\begin{equation*}
\sigma(D(P_{\varepsilon}))(x,\xi)G=|\xi|^2(JG^{\perp}+\varepsilon G^{\perp}).
\end{equation*}
It follows that
\begin{equation*}
\<\sigma(D(P_{\varepsilon}))(x,\xi)G,G \>=\<|\xi|^2(JG^{\perp}+\varepsilon G^{\perp}),G\>=\varepsilon|\xi|^2 |G^{\perp}|^2.
\end{equation*}
This shows that for each fixed $\varepsilon>0$, the operator $P_{\varepsilon}$ is weakly elliptic. The degeneracy of the equation is caused by the
diffeomorphism group of the underlying manifold, just as in the case of MCF. It is well-know that by applying the DeTurck trick, one can prove the short time
existence of a solution to the MCF.  Here we follow the same trick to show the existence of a local solution of the perturbed SMCF (\ref{e:pSMCF}).

\begin{lemma}\label{prop3.2}
For each $\varepsilon>0$, the Cauchy problem (\ref{e:pSMCF}) admits a unique smooth solution on the time interval $[0,T_{\varepsilon})$ for some $T_{\varepsilon}>0$.
\end{lemma}

\begin{proof}
First, we fix a background symmetric connection $\bar{\Ga}$ on $\Si$ and consider a modified flow
\begin{equation}\label{e:DeTurck}
  \frac{\partial \tilde{F}}{\partial t}=\tilde{J}\tilde{\textbf{H}}+\varepsilon\tilde{\textbf{H}}+\varepsilon d\tilde{F}(V),
\end{equation}
where
\[ V= \tilde{g}^{ij}(\tilde{\Ga}_{ij}^k-\bar{\Ga}_{ij}^k)\frac{\p}{\p x^k}\]
is a vector field on $\Si$, and $\tilde{g}, \tilde{\Ga}$ are the metric and connection induced by $\tilde{F}$.

In local coordinates, we have
\[(\tilde{\textbf{H}}+ d\tilde{F}(V))^\al = \tilde{g}^{ij}\(\frac{\p^2 \tilde{F}^\al}{\p x^i\p x^j}-\bar{\Ga}_{ij}^k\frac{\p \tilde{F}^\al}{\p x^k}\).\]
If we denote the right hand side of (\ref{e:DeTurck}) by $\tilde{P}_\varepsilon(\tilde{F})$, then from the computations in Section \ref{s:symbol}, we see that the principal symbol of $\tilde{P}_\varepsilon$ is
\begin{equation*}
\sigma(D(\tilde{P}_{\varepsilon}))(x,\xi)G=|\xi|^2(JG^{\perp}+\varepsilon G).
\end{equation*}
It follows that
\begin{equation*}
\<\sigma(D(\tilde{P}_{\varepsilon}))(x,\xi)G,G \>=\varepsilon|\xi|^2 |G|^2.
\end{equation*}
Hence the modified flow (\ref{e:DeTurck}) is strictly parabolic. By standard parabolic theory (see for example chapter 15 of \cite{Taylor}), we know that for any $\ep>0$ and smooth initial data, (\ref{e:DeTurck}) admits a unique smooth local solution $\tilde{F}_\ep:\Si\times[0, T_\ep]\to \Real^{n+2}$.

Next we can solve the ODE
\[\frac{\p \phi}{\p t}=V\]
to get a family of diffeomorphisms $\phi$ of $\Si$ which is generated by the vector field $V$. It is easy to show that $F_\ep:=\tilde{F}_\ep\circ \phi^{-1}$ is a solution to the original flow (\ref{app-flow}).
\end{proof}

\subsection{Evolution equations}

In this subsection, we will calculate the evolution equation of various geometric quantities for the perturbed SMCF (\ref{e:pSMCF}). Since these calculations are standard as in the case of MCF,
we only provide sketches here. Note that Lemma~\ref{l:parallel} is crucial in the calculations since we can always commute the normal connection $\n$ and the complex structure $J$.

Choose a local field of orthonormal frames $e_1, \cdots, e_n, \nu_{n+1}, \nu_{n+2}$ of $\textbf{R}^{n+2}$ along $\Sigma_t$ such that $e_1,\cdots, e_n$ are tangent vectors of $\Sigma_s$ and
$\nu_{n+1},\nu_{n+2}$ are in the normal bundle over $\Sigma_t$. We will agree on the following index ranges:
$$1\leq i,j,k,l \leq n, \ \ n+1\leq \alpha,\beta,\gamma\leq n+2, \ \ 1\leq A,B,C \leq n+2.$$
We will work on a parallel normal frame $\{\nu_{n+1},\nu_{n+2}\}$ such that $\n_t\nu_\al=0$, i.e. $\<\p_t\nu_{\alpha},\nu_{\beta}\>=0$. Since the complex structure $J$ is also parallel by Lemma~\ref{l:parallel}, we may assume that $J\nu_{n+1}=\nu_{n+2}$, $J\nu_{n+2}=-\nu_{n+1}$ for all time $t$.

For simplicity, we set
\begin{equation*}
    J\textbf{H}+\varepsilon\textbf{H}:=\textbf{V}=V^{\alpha}e_{\alpha},
\end{equation*}
which means
\begin{equation}\label{V}
    V^{n+1}=\varepsilon H^{n+1}-H^{n+2}, \ \ V^{n+2}=\varepsilon H^{n+2}+H^{n+1}.
\end{equation}

Denoting $g=(g_{ij})$ the induced metric on $\Sigma$ and $d\mu$ the induced volume form of $g$, we have
\begin{lemma}\label{lem2.1}
Along the perturbed SMCF (\ref{e:pSMCF}), we have
\begin{equation}\label{metric}
    \frac{\partial}{\partial t} g_{ij}=-2\< J\textbf{H}, \textbf{A}(e_i,e_j)\>-2\varepsilon\< \textbf{H}, \textbf{A}(e_i,e_j)\>.
\end{equation}
As a consequence, we have
\begin{equation}\label{metric2}
    \frac{\partial}{\partial t} d\mu=-\varepsilon|\textbf{H}|^2d\mu.
\end{equation}
\end{lemma}
\begin{proof}
The lemma follows by exactly the same arguments as in the proof of Lemma~\ref{l:metric}.
\end{proof}

\begin{lemma}\label{lem2.2}
Along the perturbed SMCF (\ref{e:pSMCF}), the second fundamental form satisfies
\begin{equation}\label{e2.3}
\begin{aligned}
 \frac{\partial }{\partial t}h_{ij}^{n+1}
   =& \varepsilon\Delta h^{n+1}_{ij}-\Delta h^{n+2}_{ij}+\varepsilon h^{n+1}_{im}(h^{\beta}_{mk}h^{\beta}_{kj}-h^{\beta}_{mj}H^{\beta})\\
   & +\varepsilon h^{n+1}_{mk}(h^{\beta}_{mk}h^{\beta}_{ij}-h^{\beta}_{ki}h^{\beta}_{mj})
        +\varepsilon h^{\beta}_{ik}(h^{\beta}_{kl}h^{n+1}_{lj}-h^{n+1}_{kl}h^{\beta}_{lj})\\
   & -h^{n+2}_{im}(h^{\beta}_{mk}h^{\beta}_{kj}-h^{\beta}_{mj}H^{\beta})
       -h^{n+2}_{mk}(h^{\beta}_{mk}h^{\beta}_{ij}-h^{\beta}_{ki}h^{\beta}_{mj})    \\
   & -h^{\beta}_{ik}(h^{\beta}_{kl}h^{n+2}_{lj}-h^{n+2}_{kl}h^{\beta}_{lj})
       -V^{\beta}h_{ik}^{n+1}h_{jk}^{\beta},
\end{aligned}
\end{equation}
and
\begin{equation}\label{e2.4}
\begin{aligned}
 \frac{\partial }{\partial t}h_{ij}^{n+2}
   =& \varepsilon\Delta h^{n+2}_{ij}+\Delta h^{n+1}_{ij}+\varepsilon h^{n+2}_{im}(h^{\beta}_{mk}h^{\beta}_{kj}-h^{\beta}_{mj}H^{\beta}) \\
   & +\varepsilon h^{n+2}_{mk}(h^{\beta}_{mk}h^{\beta}_{ij}-h^{\beta}_{ki}h^{\beta}_{mj})
       +\varepsilon h^{\beta}_{ik}(h^{\beta}_{kl}h^{n+2}_{lj}-h^{n+2}_{kl}h^{\beta}_{lj}) \\
   & -h^{n+1}_{im}(h^{\beta}_{mk}h^{\beta}_{kj}-h^{\beta}_{mj}H^{\beta})
       -h^{n+1}_{mk}(h^{\beta}_{mk}h^{\beta}_{ij}-h^{\beta}_{ki}h^{\beta}_{mj})    \\
   & -h^{\beta}_{ik}(h^{\beta}_{kl}h^{n+1}_{lj}-h^{n+1}_{kl}h^{\beta}_{lj})
       -V^{\beta}h_{ik}^{n+2}h_{jk}^{\beta}.
\end{aligned}
\end{equation}
In particular, we have
\begin{equation}\label{e2.5}
    \frac{\partial }{\partial t}\textbf{A}=\varepsilon\Delta \textbf{A}+J\Delta \textbf{A}+\textbf{A}*\textbf{A}*\textbf{A}.
\end{equation}
\end{lemma}

\begin{proof}
From Lemma 8.3 of \cite{AS1}, we see that
\begin{equation*}
\frac{\partial }{\partial t}h_{ij}^{n+1}=-V^{n+1}_{,ji}+V^{\beta}h_{ik}^{n+1}h_{jk}^{\beta}
 +h_{ij}^{\beta}\< \nu_{\beta}, \overline{\nabla}_{\textbf{V}}\nu_{n+1}\>,
\end{equation*}
and
\begin{equation*}
\frac{\partial }{\partial t}h_{ij}^{n+2}=-V^{n+2}_{,ji}+V^{\beta}h_{ik}^{n+2}h_{jk}^{\beta}
 +h_{ij}^{\beta}\< \nu_{\beta}, \overline{\nabla}_{\textbf{V}}\nu_{n+2}\>.
\end{equation*}
Since by our choice of the normal frame, we have
$$\n_\textbf{V}\nu_\alpha = \n_t\nu_\alpha =0,$$
both of the last terms in the above two identities will disappear.
As for the term $V^{\alpha}_{,ji}$ where $V$ is given by (\ref{V}), recall the following commutation formula (see for example Page 332 of \cite{Wang})
\begin{equation*}
    \Delta h^{\alpha}_{ij}=H^{\alpha}_{,ij}+h^{\alpha}_{im}(h^{\gamma}_{mj}H^{\gamma}-h^{\gamma}_{mk}h^{\gamma}_{kj})
       +h^{\alpha}_{mk}(h^{\gamma}_{mj}h^{\gamma}_{ik}-h^{\gamma}_{mk}h^{\gamma}_{ij})
       +h^{\beta}_{ik}(h^{\beta}_{lj}h^{\alpha}_{lk}-h^{\beta}_{lk}h^{\alpha}_{lj}).
\end{equation*}
Then (\ref{e2.3}) and (\ref{e2.4}) follows by direct computation. Furthermore, since
\begin{equation*}
    (J\Delta\textbf{A})_{ij}=-\Delta h^{n+2}_{ij}\nu_{n+1}+\Delta h^{n+1}_{ij}\nu_{n+2},
\end{equation*}
(\ref{e2.5}) follows easily from (\ref{e2.3}) and (\ref{e2.4}).
\end{proof}

\begin{lemma}
Along the perturbed SMCF (\ref{e:pSMCF}), we have
\begin{equation}\label{e2.6}
    \frac{\partial }{\partial t}|\textbf{A}|^2=\varepsilon\Delta |\textbf{A}|^2-2\epsilon|\nabla \textbf{A}|^2
    +2\< J\Delta \textbf{A},\textbf{A}\>+\textbf{A}*\textbf{A}*\textbf{A}*\textbf{A}.
\end{equation}
In particular, we have
\begin{equation}\label{e2.7}
\frac{d}{dt}\int_{\Sigma}|\textbf{A}|^2d\mu\leq
         -2\epsilon\int_{\Sigma}|\nabla\textbf{A}|^2d\mu-\epsilon\int_{\Sigma}|\textbf{H}|^2|\textbf{A}|^2d\mu
                 +C(n)\int_{\Sigma}|\textbf{A}|^4d\mu.
\end{equation}
\end{lemma}

\begin{proof}
From (\ref{metric}), we see that $\frac{\partial}{\partial t}g_{ij}=\textbf{A}*\textbf{A}$, which implies that $\frac{\partial}{\partial t}g^{ij}=\textbf{A}*\textbf{A}$. Therefore, using
(\ref{e2.5}), we compute
\begin{eqnarray*}
 \frac{\partial }{\partial t}|\textbf{A}|^2
   &=& \frac{\partial }{\partial t}\left(g^{ik}g^{jl}h^{\alpha}_{ij}h^{\alpha}_{jl}\right)
        =\textbf{A}*\textbf{A}*\textbf{A}*\textbf{A}+2\< \textbf{A},\frac{\partial}{\partial t}\textbf{A}\>\\
   &=& 2\varepsilon\< \textbf{A},\Delta \textbf{A}\>+2\< J\Delta \textbf{A},\textbf{A}\>
        +\textbf{A}*\textbf{A}*\textbf{A}*\textbf{A}\\
   &=& \varepsilon\Delta |\textbf{A}|^2-2\epsilon|\nabla \textbf{A}|^2
    +2\< J\Delta \textbf{A},\textbf{A}\>+\textbf{A}*\textbf{A}*\textbf{A}*\textbf{A}.
\end{eqnarray*}
Furthermore, from (\ref{e2.6}) and (\ref{metric2}), we have
\begin{eqnarray*}
\frac{d}{dt}\int_{\Sigma}|\textbf{A}|^2d\mu
   &=& \int_{\Sigma}\left(\varepsilon\Delta |\textbf{A}|^2-2\epsilon|\nabla \textbf{A}|^2
    +2\< J\Delta \textbf{A},\textbf{A}\>+\textbf{A}*\textbf{A}*\textbf{A}*\textbf{A}-\varepsilon|\textbf{H}|^2|\textbf{A}|^2\right)d\mu\\
   &=& \int_{\Sigma}\left(-2\epsilon|\nabla \textbf{A}|^2-2\< J\nabla \textbf{A},\nabla\textbf{A}\>
      +\textbf{A}*\textbf{A}*\textbf{A}*\textbf{A}-\varepsilon|\textbf{H}|^2|\textbf{A}|^2\right)d\mu\\
   &=& \int_{\Sigma}\left(-2\epsilon|\nabla \textbf{A}|^2
      +\textbf{A}*\textbf{A}*\textbf{A}*\textbf{A}-\varepsilon|\textbf{H}|^2|\textbf{A}|^2\right)d\mu\\
   &\leq& -2\epsilon\int_{\Sigma}|\nabla\textbf{A}|^2d\mu-\epsilon\int_{\Sigma}|\textbf{H}|^2|\textbf{A}|^2d\mu
                 +C(n)\int_{\Sigma}|\textbf{A}|^4d\mu.
\end{eqnarray*}
Here, we have used Lemma \ref{l:parallel} and the fact that the complex structure $J$ is skew-symmetric.
\end{proof}

In order to get the evolution equation for derivatives of the second fundamental form, we need the following commutation formulas (see Lemma 3.2 of \cite{HanSun}).

\begin{lemma}
Suppose $g_{t}$ is a family of metric on $\Sigma$ satisfying
$\frac{\partial g_{t}}{\partial t}=h$. Let $\Delta$ and $\nabla$ be
the Laplacian and connection induced by $g_{t}$. Then for any tensor
$S$ on $\Sigma$, we have
\begin{equation}\label{e2.9}
\frac{\partial}{\partial t}\nabla S- \nabla \frac{\partial}{\partial
 t}S=S*\nabla h,
\end{equation}
\begin{equation}\label{e2.10}
\nabla (\Delta S)- \Delta(\nabla S)=\nabla Rm *S+Rm*\nabla S.
\end{equation}
Here $Rm$ is the curvature tensor on $\Sigma$.
\end{lemma}

Then we can prove the following lemma by induction.

\begin{lemma}\label{prop2.4}
Along the perturbed SMCF (\ref{e:pSMCF}), we have for any integer $l\geq 0$,
\begin{equation}\label{e2.11}
    \frac{\partial}{\partial t}\nabla^{l}\textbf{A}=\varepsilon\Delta\nabla^{l}\textbf{A}+J\Delta\nabla^{l}\textbf{A}
                 +\sum_{i+j+k=l}\nabla^{i}\textbf{A}*\nabla^{j}\textbf{A}*\nabla^{k}\textbf{A}.
\end{equation}
As a consequence, we have
\begin{eqnarray}\label{e2.12}
 \frac{\partial}{\partial t}|\nabla^{l}\textbf{A}|^2
   &\leq& \varepsilon\Delta|\nabla^{l}\textbf{A}|^2-2\epsilon|\nabla^{l+1} \textbf{A}|^2
                +\< J\Delta\nabla^{l}\textbf{A},\nabla^{l}\textbf{A}\> \nonumber\\
   & &  +c(n,l)\sum_{i+j+k=l}|\nabla^{i}\textbf{A}|\cdot|\nabla^{j}\textbf{A}|\cdot|\nabla^{k}\textbf{A}|\cdot|\nabla^{l}\textbf{A}|,
\end{eqnarray}
where $c(n, l)$ is a constant depending on $n$ and $l$.
\end{lemma}

\begin{proof}
By (\ref{metric}), we see that $h=\frac{\partial g}{\partial t}=\textbf{A}*\textbf{A}$, so that $\nabla h=\nabla\textbf{A}*\textbf{A}$. Applying (\ref{e2.9}) inductively , we get that
\begin{eqnarray*}
\frac{\partial}{\partial t}\nabla^{l}\textbf{A}
   &=& \nabla\frac{\partial}{\partial t}\nabla^{l-1}\textbf{A}+\nabla^{l-1}\textbf{A}*\nabla\textbf{A}*\textbf{A} \\
   &=& \nabla\left(\nabla\frac{\partial}{\partial t}\nabla^{l-2}\textbf{A}+\nabla^{l-2}\textbf{A}*\nabla\textbf{A}*\textbf{A}\right)
            +\nabla^{l-1}\textbf{A}*\nabla\textbf{A}*\textbf{A} \\
   &=& \nabla^l\frac{\partial}{\partial t}\textbf{A}+\sum_{i+j+k=l}\nabla^{i}\textbf{A}*\nabla^{j}\textbf{A}*\nabla^{k}\textbf{A}\\
   &=& \varepsilon\nabla^l\Delta\textbf{A}+J\nabla^l\Delta \textbf{A}+\sum_{i+j+k=l}\nabla^{i}\textbf{A}*\nabla^{j}\textbf{A}*\nabla^{k}\textbf{A}.
\end{eqnarray*}
Here in the last equality, we have used (\ref{e2.5}) and Lemma \ref{l:parallel}.

Next, note that by Gauss equation, the curvature tensor $Rm$ on $\Sigma$ can be expressed as $Rm=\textbf{A}*\textbf{A}$, so that $\nabla Rm=\nabla\textbf{A}*\textbf{A}$. Then inductively
applying (\ref{e2.10}), we get that
\begin{eqnarray*}
\nabla^{l}\Delta\textbf{A}
   &=&\nabla^{l-1}\left(\Delta\nabla\textbf{A}+\nabla\textbf{A}*\textbf{A}*\textbf{A}\right) \\
   &=& \nabla^{l-2}\left(\Delta\nabla^2\textbf{A}+\nabla\textbf{A}*\nabla\textbf{A}*\textbf{A}+\nabla^2\textbf{A}*\textbf{A}*\textbf{A}\right)
            +\sum_{i+j+k=l}\nabla^{i}\textbf{A}*\nabla^{j}\textbf{A}*\nabla^{k}\textbf{A} \\
   &=& \nabla^{l-2}\Delta\nabla^2\textbf{A}+\sum_{i+j+k=l}\nabla^{i}\textbf{A}*\nabla^{j}\textbf{A}*\nabla^{k}\textbf{A} \\
   &=& \cdots \\
   &=& \Delta\nabla^{l}\textbf{A}+\sum_{i+j+k=l}\nabla^{i}\textbf{A}*\nabla^{j}\textbf{A}*\nabla^{k}\textbf{A}.
\end{eqnarray*}
Combining the above two equalities together gives us (\ref{e2.11}). Then (\ref{e2.12}) follows easily.
\end{proof}

The next inequality is a direct corollary of Lemma \ref{lem2.1} and Lemma \ref{prop2.4}.

\begin{lemma}
Along the perturbed SMCF (\ref{e:pSMCF}), we have
\begin{equation}\label{e2.13}
\frac{d}{dt}\int_{\Sigma}|\nabla^{l}\textbf{A}|^2d\mu \le
c(n,l)\sum_{i+j+k=l}\int_{\Sigma}|\nabla^{i}\textbf{A}|\cdot|\nabla^{j}\textbf{A}|\cdot|\nabla^{k}\textbf{A}|\cdot|\nabla^{l}\textbf{A}|d\mu.
\end{equation}
\end{lemma}

Next, let's recall the following interpolation inequality proved by Hamilton (\cite{Ha}, Section 12)

\begin{lemma}\label{lemma-interpo}
If $T$ is any tensor and if $1\leq i\leq l-1$, then with a constant $C=C(n,l)$ depending only on $n=\mbox{dim} \Sigma$ and $l$, which is independent of the metric $g$ and the connection
$\Gamma$, we have the estimate
\begin{equation*}
\int_{\Sigma}|\nabla^i T|^{\frac{2l}{i}}d\mu \leq C \max_{\Sigma}|T|^{2\left(\frac{l}{i}-1\right)}\int_{\Sigma}|\nabla^l T|^2d\mu.
\end{equation*}
\end{lemma}

Finally, by using (\ref{e2.13}) and Lemma \ref{lemma-interpo} in the same way as in Section 7 of \cite{Hu1}, we may obtain

\begin{lemma}\label{Lemma7.1}
Along the perturbed SMCF (\ref{e:pSMCF}), we have
\begin{equation}\label{E7.1}
\frac{d}{dt}
\int_{\Sigma(t)}|\nabla^{l}\textbf{A}|^2d\mu \leq   c(n,l)\max_{\Sigma}|\textbf{A}|^2
         \int_{\Sigma}|\nabla^{l}\textbf{A}|^2d\mu.
\end{equation}
\end{lemma}

\subsection{Proof of the main theorem}\label{s:existence}

Now we come back to the case of two dimensional SMCF in $\Real^4$ and finish the proof of local existence of SMCF.

\begin{proof}[Proof of Theorem \ref{t:main}]
By Lemma \ref{prop3.2}, we know that for each $\ep>0$, there exists a smooth solution $F_\ep$ to (\ref{app-flow}) on a maximal time interval
$[0,T_{\ep})$. For convenience, we denote the second fundamental form of $F_\ep$ at time $t$ by $\A_\ep(t)$.

For each $0<\ep<1/4$, we define a time $T'_\ep\in [0,T_\ep]$ by the maximal time such that for all $t\in [0,T'_\ep)$,
\begin{equation*}
    \norm{\textbf{A}_{\varepsilon}(t)}_{H^{2,2}}\leq 2\norm{\textbf{A}_0}_{H^{2,2}}:= B.
\end{equation*}
Obviously $T'_\ep$ is positive since $A_\ep(t)$ is smooth on $t$. In fact, we will show that there is a uniform positive lower bound for $T'_\ep$.

To see this, we first note that by Lemma~\ref{lem2.1}, for all $t\in[0,T^{'}_\varepsilon)$ the volume of $\Si_\ep(t):=F_\ep(t, \Si)$ satisfies
\begin{equation*}
\frac{d}{dt}|\Si_\ep|=-\ep\int_{\Si_\ep}|\textbf{H}_\ep|^2d\mu_\ep\ge -2\ep\norm{\textbf{A}_\ep(t)}_{L^2}^2\ge-\frac{1}{2}B^2.
\end{equation*}
Thus there exists a uniform time $T_1:= |\Si_0|/B^2$, such that for all $t\in [0,T_\ep')\cap [0,T_1]$
\begin{equation}\label{E7.3}
|\Si_0|\ge |\Si_\ep(t)|\ge |\Si_0|-\frac{1}{2}B^2T_1=\frac{|\Si_0|}{2}:=m.
\end{equation}

Now if $T'_\ep\ge T_1$, then we already have a lower bound. Thus we may assume $T'_\ep<T_1$ and in this case, we claim that $T'_\ep<T_\ep$.

First we assume that the claim is true, then clearly by the definition of $T'_\ep$ we have
\[ \norm{\A_\ep(T'_\ep)}_{H^{2,2}}=2\norm{\A_0}_{H^{2,2}}.\]
Applying Theorem~\ref{t:sobolev}, we have a uniform $C^0$ bound of the second fundamental form
\begin{equation}\label{E.c0}
\norm{\textbf{A}_\ep(t)}_{C^0}\le C(B,m)
\end{equation}
on the time interval $[0,T_\ep')$.
It follows from Lemma \ref{Lemma7.1} that for any integer $l\ge 0$,
\begin{equation*}
\frac{d}{dt}
\int_{\Sigma_{\varepsilon}(t)}|\nabla^{l}\textbf{A}_{\varepsilon}(t)|^2d\mu \leq   c(2,l)C(B,m) ^2    \int_{\Sigma_{\varepsilon}(t)}|\nabla^{l}\textbf{A}_{\varepsilon}(t)|^2d\mu.
\end{equation*}
Consequently, by Gronwall's inequality, we have
\begin{equation}\label{E7.19}
\norm{\textbf{A}_{\varepsilon}(t)}_{H^{l,2}}\le e^{c_lC(B,m)^2t}\norm{\textbf{A}_0}_{H^{l,2}},
\end{equation}
where $c_l:=\max\{c(2,0), \cdots, c(2,l)\}$ only depends on $l$.

Setting $t=T^{'}_{\varepsilon}$ and $l=2$ in (\ref{E7.19}) yields
\begin{equation*}
\norm{\A_\ep(T'_\ep)}_{H^{2,2}}=2\norm{\textbf{A}_0}_{H^{2,2}}\leq e^{c_2C(B,m)^2T^{'}_{\varepsilon}}\norm{\textbf{A}_0}_{H^{2,2}}.
\end{equation*}
It follows that
\begin{equation*}
T^{'}_{\varepsilon}\geq \frac{\log2}{c_2C(B,m)^2}:= T_2.
\end{equation*}
Therefore, we get a uniform lower bound for $T^{'}_{\varepsilon}$ given by $T_0:=\min\{T_1, T_2\}$, which is decided by $\norm{\textbf{A}_0}_{H^{2,2}}$ and $|\Si_0|$.

Next we restrict ourselves on the time span $[0,T_0]$. For any $\varepsilon\in (0,1/4)$, we have uniform bounds of the volume $|\Si_\ep(t)|$ by (\ref{E7.3}) and the $C^0$-norm of
$\textbf{A}_\ep(t)$ by (\ref{E.c0}). Moreover, if $F_0\in C^{\infty}$ and hence $\textbf{A}_0\in C^{\infty}$, we have uniform bounds on $H^{l,2}$-norm of $\textbf{A}_\ep(t)$ for any $l\geq0$ by
(\ref{E7.19}). Then Lemma \ref{lemma-interpo} yields
\begin{equation*}
\int_{\Sigma_{\varepsilon}(t)}|\nabla^{k}\textbf{A}_{\varepsilon}(t)|^pd\mu\leq C(k,p),
\end{equation*}
for all $k\geq 0$ and $p>0$. Then a version of Michael-Simon inequality (see, for example, Theorem 5.6 of \cite{KS1}) implies
\begin{equation}\label{e:uni}
\norm{\textbf{A}_\ep(t)}_{C^k}\leq C(k),
\end{equation}
for any $k\geq 0$.

It follows from (\ref{e:uni}) and standard arguments (cf. \cite{KS1}, Section 4) that in every local chart, we have
\begin{equation*}
\norm{\partial ^k F_\ep(t)}_{\infty}, \ \norm{\partial ^k\partial_t F_\ep(t)}_{\infty} \leq C(k,F_0),
\end{equation*}
for  any $k\geq0$, where $\partial$ is the partial derivatives in the local charts. Then by Arzela-Ascoli Theorem, we conclude that there is a sub-sequence $\ep_i\to 0$ such that $F_{\ep_i}$ converging smoothly
to a limit $F_{\infty}\in C^\infty([0,T_0]\times \Si)$.
By taking $\ep_i \to 0$ in (\ref{app-flow}), it easy to verify that $F_{\infty}$ is a smooth solution to the SMCF (\ref{e:SMCF1}).

Finally, it remains to prove our claim on $T'_\ep$. We argue by contradiction and suppose $T'_\ep=T_\ep$ is the maximal existence time. Then by repeating the above arguments for $F_\ep$ on $[0,T_\ep)$, we see that in every local chart we have uniform bounds of $\norm{\partial ^k F_\ep(t)}_{\infty}$ and $\norm{\partial ^k\partial_t F_\ep(t)}_{\infty}$. It follows that $F_\ep(t)$ converges smoothly to an immersion as $t\to T_\ep$. By Lemma~\ref{prop3.2}, the flow can be continued for another positive time interval. This, however, contradicts with the definition of $T_\ep$.
\end{proof}

\begin{remark}
With similar arguments, it is easy to show that the maximal existence time of the local solution $F_\infty$ is characterized by the first time $T$ such that
\[ \lim_{t\to T}\norm{\textbf{A}(t)}_{H^{2,2}}=\infty.\]
\end{remark}

\end{document}